\documentclass[11pt]{amsart}

\usepackage{amssymb, amscd, txfonts}
\usepackage{graphicx}

\numberwithin{equation}{section}

\newtheorem{theorem}{Theorem}[section]
\newtheorem{proposition}[theorem]{Proposition}
\newtheorem{lemma}[theorem]{Lemma}
\newtheorem{corollary}[theorem]{Corollary}

\theoremstyle{definition}

\theoremstyle{remark}

\newtheorem{maintheorem}{\rm {\bf Theorem}}

\renewcommand{\hom}{\operatorname{Hom}}
\renewcommand{\ker}{\operatorname{Ker}}

\newcommand{\Z}{\mathbb{Z}}

\newcommand{\C}{\mathbb{C}}

\newcommand{\proj}{{\mathbb P}}

\newcommand{\SL}{{\rm SL}}
\newcommand{\GL}{{\rm GL}}
\newcommand{\PGLPGL}{{\rm PGL}_2\times{\rm PGL}_2}
\newcommand{\SLSL}{{\rm SL}_2\times{\rm SL}_2}
\newcommand{\GLGL}{{\rm GL}_2\times{\rm GL}_2}
\newcommand{\Oline}{\mathcal{O}_{{\mathbb P}^{1}}}

\newcommand{\sheaf}{\mathcal{O}}

\begin{document}

\title[]{The rationality of the moduli spaces of trigonal curves}
\author[]{Shouhei Ma}
\thanks{Supported by Grant-in-Aid for Scientific Research (S), No 22224001.} 
\address{Graduate~School~of~Mathematics, Nagoya~University, Nagoya 464-8604, Japan}
\email{ma@math.nagoya-u.ac.jp}
\subjclass[2000]{Primary 14H45, Secondary 14H10, 14E08, 14L30}
\keywords{trigonal curve, rationality, ${\SLSL}$, bi-transvectant} 
\maketitle 

\begin{abstract}
The moduli spaces of trigonal curves are proven to be rational when the genus is divisible by $4$. 
\end{abstract}

\maketitle


\section{Introduction}\label{sec: intro}

A smooth projective curve is called \textit{trigonal} if it carries a free $g_3^1$. 
When the curve has genus $\geq5$, such a pencil is unique if it exists. 
The object of our study is the moduli space $\mathcal{T}_g$ of trigonal curves of genus $g\geq5$. 
This space has been proven to be rational when $g\equiv2 \; (4)$ by Shepherd-Barron \cite{SB1}, 
and when $g$ is odd in \cite{Ma1}. 
In the present article we prove that $\mathcal{T}_g$ is rational in the remaining case $g\equiv0 \; (4)$, 
completing the following. 


\begin{maintheorem}
The moduli space $\mathcal{T}_g$ of trigonal curves of genus $g$ is rational for every $g\geq5$. 
\end{maintheorem}

$\mathcal{T}_g$ is naturally regarded as a sublocus of 
the moduli space $\mathcal{M}_g$ of genus $g$ curves. 
The rationality of $\mathcal{T}_g$ can be seen as an extension of that of the hyperelliptic locus 
due to Katsylo and Bogomolov \cite{Ka1}, \cite{B-K}. 
It would be interesting whether the tetragonal and pentagonal loci are rational as well. 
They are unirational (see, e.g., \cite{A-C}, \cite{Sc}), 
but at present known to be rational only for tetragonal of genus $7$ (\cite{B-B-C}). 
A related question is whether one can find a rational locus in $\mathcal{M}_g$ of larger dimension. 
When $g\geq23$, Castorena and Ciliberto \cite{C-C} show that 
$\mathcal{T}_g$ has larger dimension than any other locus that is (generically) 
the natural image of a linear system on a surface. 
Thus, for the above question, one would next look at 
curves in a variety of dimension $\geq3$ whose ideals have simple description. 
Note that tetragonal and pentagonal curves can be constructed in such ways (\cite{Sc}). 

We approach our problem from invariant theory for ${\SLSL}$. 
Let $V_{a,b}=H^0({\sheaf}_{{\proj}^1\times{\proj}^1}(a, b))$ be the space of bi-forms of bidegree $(a, b)$ on ${\proj}^1\times{\proj}^1$, 
which is an irreducible representation of ${\SLSL}$. 
It is classically known that a general trigonal curve $C$ of genus $g=4N$ is canonically embedded in ${\proj}^1\times{\proj}^1$ 
as a smooth curve of bidegree $(3, 2N+1)$. 
This is based on the fact that the canonical model of $C$ lies on a unique rational normal scroll 
which is isomorphic to ${\proj}^1\times{\proj}^1$. 
As a consequence, we have a natural birational equivalence 
\begin{equation*}
\mathcal{T}_{4N} \sim {\proj}V_{3, 2N+1}/{\SLSL}. 
\end{equation*}
Hence the problem is restated as follows. 

\begin{theorem}\label{main}
The quotient ${\proj}V_{3,b}/{\SLSL}$ is rational for every odd $b\geq5$. 
\end{theorem}

To prove this, we adopt the traditional and computational method of 
\textit{double bundle} (\cite{B-K}, \cite{SB2}) as follows. 
By examining the Clebsch-Gordan formula for ${\SLSL}$, 
we take a suitable ${\SLSL}$-bilinear mapping (bi-transvectant) 
\begin{equation*}\label{eqn:transvectant intro}
T : V_{3,b}\times V_{a',b'} \to V_{a'',b''} 
\end{equation*}
such that ${\dim}V_{a',b'}>{\dim}V_{a'',b''}$. 
Put $c={\dim}V_{a',b'}-{\dim}V_{a'',b''}$ and 
let $G(c, V_{a',b'})$ be the Grassmannian of $c$-dimensional subspaces of $V_{a',b'}$. 
Then $T$ induces the rational map 
\begin{equation}\label{eqn:double bundle intro}
V_{3,b} \dashrightarrow G(c, V_{a',b'}), \qquad v\mapsto{\ker}(T(v, \cdot)). 
\end{equation}
We shall find a bi-transvectant for which 
\eqref{eqn:double bundle intro} is well-defined as a rational map and is dominant. 
In that case, \eqref{eqn:double bundle intro} makes $V_{3,b}$ birationally 
an ${\SLSL}$-linearized vector bundle over $G(c, V_{a',b'})$. 
Utilizing this bundle structure and taking care of $-1$ scalar action, 
we reduce the rationality of ${\proj}V_{3,b}/{\SLSL}$ to a stable rationality of $G(c, V_{a',b'})/{\SLSL}$, 
which in turn can be shown in a more or less standard way. 

The point for this proof is to choose the bi-transvectant $T$ carefully so that 
(i) $a', b', c$ are odd (to care $-1$ scalar action) and that 
(ii) $c$ is small (for $V_{3,b}$ to have larger dimension than $G(c, V_{a',b'})$). 
For that, we will provide $T$ according to the remainder of $b$ modulo $5$, 
based on some easy calculation. 
Then the bulk of proof is devoted to verifying non-degeneracy of \eqref{eqn:double bundle intro}, 
which is facilitated by keeping $c$ small but is still somewhat laborious. 

The rest of the article is as follows. 
In \S \ref{ssec:transvectant} we recall bi-transvectants. 
We explain the method of double bundle in \S \ref{ssec:double bundle}. 
In \S \ref{sec:stable rationality} we prepare some stable rationality results in advance, 
to which the rationality of ${\proj}V_{3,b}/{\SLSL}$ will be eventually reduced. 
Then we prove Theorem \ref{main} in \S \ref{sec:proof}. 

We work over the complex numbers. 
The Grassmannian $G(a, V)$ parametrizes $a$-dimensional linear subspaces of the vector space $V$. 


\vspace{0.3cm}
\noindent
\textbf{Acknowledgement.}
I would like to thank the referees 
for valuable suggestions on the presentation of the manuscript.


\section{Bi-transvectant}\label{sec:transvectant}


\subsection{Bi-transvectant}\label{ssec:transvectant}

We write $V_d$ for the ${\SL}_2$-representation $H^0({\Oline}(d))$, 
the space of binary forms of degree $d$. 
Let $e\leq d$. 
According to the Clebsch-Gordan decomposition 
\begin{equation}\label{eqn:C-G SL2}
V_d\otimes V_e = \bigoplus_{r=0}^{e}V_{d+e-2r}, 
\end{equation}
there exists a unique (up to constant) ${\SL}_2$-bilinear mapping  
\begin{equation*}
T^{(r)}:V_d\times V_e\to V_{d+e-2r}, 
\end{equation*}
which is called the \textit{$r$-th transvectant}. 
For two binary forms $F(X, Y)\in V_d$ and $G(X, Y)\in V_e$, 
we have the well-known explicit formula (cf. \cite{P-V}) 
\begin{equation}\label{eqn:transvectant}
T^{(r)}(F, G) =  \frac{(d-r)!}{d!}\frac{(e-r)!}{e!}
                        \sum_{i=0}^{r} (-1)^{i} \binom{r}{i}
                         \frac{\partial^rF}{\partial X^{r-i}\partial Y^i}\frac{\partial^rG}{\partial X^{i}\partial Y^{r-i}}. 
\end{equation}
We will need this formula when $r=e$ and $r=e-1$.

The $e$-th transvectant   
$T^{(e)}\colon V_d\times V_e \to V_{d-e}$
is especially called the \textit{apolar covariant}. 
By \eqref{eqn:transvectant}, $T^{(e)}(F, G)$ is calculated by applying the differential polynomial 
$(d!)^{-1}(d-e)!G(-\partial_Y, \partial_X)$ to $F(X, Y)$. 
In particular, we have 
\begin{equation*}\label{eqn:apolar}
T^{(e)}(X^iY^{d-i}, X^{e-j}Y^j) = \left\{ \begin{array}{cl} 
                             (-1)^{e-j}\binom{d}{i}^{-1}\binom{d-e}{i-j}X^{i-j}Y^{(d-e)-(i-j)},   &      j\leq i, \; e-j\leq d-i,  \\
                             0,                   &            \text{otherwise.}      \\
                               \end{array} \right. 
\end{equation*}
For the $(e-1)$-th transvectant 
$T^{(e-1)}\colon V_d\times V_e \to V_{d-e+2}$, 
we have 
\begin{equation*}
T^{(e-1)}( \cdot , X^{e-j}Y^{j}) = (-1)^{e-j}\frac{1}{e}\frac{(d-e+1)!}{d!}
                                         \left\{ jY\partial_{X}^{j-1}\partial_{Y}^{e-j} - (e-j)X\partial_{X}^{j}\partial_{Y}^{e-j-1} \right\}, 
\end{equation*}
where $\partial_{X}^{-1}=\partial_{Y}^{-1}=0$ by convention. 
Therefore                                           
\begin{equation*}\label{eqn:pre-apolar}
T^{(e-1)}(X^iY^{d-i}, X^{e-j}Y^j) = 
                            \left\{ \begin{array}{cl} 
                             A X^{i-j+1}Y^{(d-i)-(e-j)+1},   &  \;    j\leq i+1, \; e-j\leq d-i+1,  \\ 
                             0,                   &    \;        \text{otherwise},      \\
                               \end{array} \right. 
\end{equation*}
where 
\begin{equation*}
A = (-1)^{e-j}\binom{d}{i}^{-1}\binom{d-e+2}{i-j+1} \frac{j(d+2)-(i+1)e}{e(d-e+2)}. 
\end{equation*}
We stress in particular that 

\begin{lemma}\label{pre-apolar non-dege}
Let $0\leq j\leq i+1$ and $0\leq e-j\leq d-i+1$. 
The bilinear map 
\begin{equation*}
T^{(e-1)} : {\C}X^iY^{d-i}\times{\C}X^{e-j}Y^j \to {\C}X^{i-j+1}Y^{(d-e+2)-(i-j+1)} 
\end{equation*}
is non-degenerate if and only if $j(d+2)\ne(i+1)e$. 
This is always the case when $d+2$ is coprime to $e$. 
\end{lemma}

Now we consider ${\SLSL}$-representations. 
The space $V_{a,b}=H^0({\sheaf}_{{\proj}^1\times{\proj}^1}(a, b))$ is 
the tensor representation $V_a\boxtimes V_b$. 
Substituting \eqref{eqn:C-G SL2} into 
\begin{equation*}
V_{a,b}\otimes V_{a',b'} = (V_a\otimes V_{a'})\boxtimes(V_b\otimes V_{b'}), 
\end{equation*}
we obtain the Clebsch-Gordan decomposition for ${\SLSL}$, 
\begin{equation*}\label{eqn: C-G SL2SL2}
V_{a,b}\otimes V_{a',b'} = \bigoplus_{r, s}V_{a+a'-2r, b+b'-2s},  
\end{equation*}
where $0\leq r\leq{\min}\{a, a'\}$ and $0\leq s\leq{\min}\{b, b'\}$. 
To each irreducible summand $V_{a+a'-2r, b+b'-2s}$ is associated the \textit{$(r, s)$-th bi-transvectant} 
\begin{equation*}\label{eqn: transvectant}
T^{(r,s)} : V_{a,b}\times V_{a',b'} \to V_{a+a'-2r, b+b'-2s}. 
\end{equation*}
This ${\SLSL}$-bilinear mapping is calculated from the transvectants by 
\begin{equation*}\label{eqn: transvectant II}
T^{(r,s)}(F\boxtimes G, \; F'\boxtimes G') = T^{(r)}(F, F')\boxtimes T^{(s)}(G, G'), 
\end{equation*}
where 
$F\in V_a$, $G\in V_b$, $F'\in V_{a'}$, and $G'\in V_{b'}$.


\subsection{The method of double bundle}\label{ssec:double bundle}

In \S \ref{sec:proof}, we will use the method of double bundle (\cite{B-K}) and its generalization (\cite{SB2}). 
We here give some account in the present situation. 
The strategy is to find a certain bi-transvectant 
which introduces on the target ${\proj}V_{a,b}$ a fibration structure over a Grassmannian, 
and then reduce the rationality of ${\proj}V_{a,b}/{\SLSL}$ to a stable rationality of the quotient of the latter.

Suppose we have a bi-transvectant 
\begin{equation*}
T=T^{(r,s)} : V_{a,b}\times V_{a',b'} \to V_{a'',b''} 
\end{equation*}
such that 
$c:={\dim}V_{a',b'}-{\dim}V_{a'',b''}$ is positive and that 
${\dim}V_{a,b}>c\cdot{\dim}V_{a'',b''}$. 
This bilinear map induces an ${\SLSL}$-linear embedding 
\begin{equation*}
 V_{a,b} \subset {\hom}(V_{a',b'}, V_{a'',b''}). 
\end{equation*}
The space ${\hom}(V_{a',b'}, V_{a'',b''})$ is birationally fibered over $G(c, V_{a',b'})$, 
by sending a surjective linear map to its kernel. 
We can thus consider an ${\SLSL}$-equivariant rational map 
\begin{equation*}\label{eqn:double bundle general}
\varphi : V_{a,b} \dashrightarrow G(c, V_{a',b'}), \qquad v\mapsto{\ker}(T(v, \cdot)). 
\end{equation*}
We assume (hope) that 
\begin{equation}\tag{$\clubsuit$}\label{eqn:non-degeneracy}
\textit{$\varphi$ is defined on a non-empty locus, and is dominant.}
\end{equation}
This means that the position of $V_{a,b}$ inside ${\hom}(V_{a',b'}, V_{a'',b''})$ is 
"non-degenerate" with regards to the fibration over $G(c, V_{a',b'})$. 
The inequality ${\dim}V_{a,b}>c\cdot{\dim}V_{a'',b''}$ above 
is the dimension condition necessary for the dominance of $\varphi$ to be possible. 
If \eqref{eqn:non-degeneracy} holds, 
then $V_{a,b}$ becomes birational to the unique component $\mathcal{E}$ of the incidence 
\begin{equation*}\label{incidence}
\mathcal{X} = \{ (v, P) \in V_{a,b}\times G(c, V_{a',b'}), \: \: T(v, P)\equiv0 \}
\end{equation*}
that dominates $G(c, V_{a',b'})$. 
Indeed, the first projection $\pi\colon\mathcal{X}\to V_{a,b}$ is isomorphic over the domain $U$ of regularity of $\varphi$, 
and then the dominance of $\varphi$ implies that $\pi^{-1}(U)$ is contained in $\mathcal{E}$. 
Since $\mathcal{E}$ is (generically) a sub vector bundle of $V_{a,b}\times G(c, V_{a',b'})$ 
preserved under the ${\SLSL}$-action, 
it is an ${\SLSL}$-linearized vector bundle over $G(c, V_{a',b'})$. 
We shall then try to apply the following no-name lemma (cf.~\cite{Do}). 

\begin{lemma}[no-name lemma]
Let $G$ be an algebraic group and $\mathcal{E}\to X$ a $G$-linearized vector bundle of rank $N+1$. 
Suppose that $G$ acts on $X$ almost freely. Then 
\begin{equation*}
{\proj}\mathcal{E} / G \sim {\proj}^N \times (X/G). 
\end{equation*}
\end{lemma}

In the present situation, however, ${\SLSL}$ never acts on $G(c, V_{a',b'})$ almost freely 
because of the presence of $(\pm1, \mp1)\in{\SLSL}$. 
So we should take $G={\PGLPGL}$, whose action on $G(c, V_{a',b'})$ is now almost free in most cases, 
but then the ${\SLSL}$-linearization on $\mathcal{E}$ may not descends to that of $G$. 
To deal with this problem, 
we want to tensor $\mathcal{E}$ with an ${\SLSL}$-linearized line bundle $\mathcal{L}$ 
that kills the action of $(\pm1, \mp1)$ on $\mathcal{E}$. 
If this was successful, we would have 
\begin{equation}\label{eqn:apply twisted no-name}
{\proj}\mathcal{E}/G = {\proj}(\mathcal{E}\otimes\mathcal{L})/G \sim {\proj}^N \times (G(c, V_{a',b'})/G) 
\end{equation}
where $N={\dim}\,{\proj}V_{a,b} - {\dim}\, G(c, V_{a',b'})$. 
Thus the rationality of ${\proj}V_{a,b}/G$ could be reduced to a stable rationality of $G(c, V_{a',b'})/G$, 
which is much easier to prove: 
we prepare results of this sort in the next \S \ref{sec:stable rationality}.

In practice, we will check the non-degeneracy requirement \eqref{eqn:non-degeneracy} as follows. 

\begin{lemma}[cf.\;\cite{B-K}]\label{non-degeneracy}
The condition \eqref{eqn:non-degeneracy} is satisfied if and only if 
there exist vectors $v \in V_{a,b}$ and $w_1,\cdots,w_c \in V_{a',b'}$ such that 

(i) $w_1,\cdots,w_c$ are linearly independent, 

(ii) $T(v, w_i)=0$ for every $w_i$, 

(iii) the map $T(v, \cdot) : V_{a',b'} \to V_{a'',b''}$ is surjective, and 

(iv) the map $(T(\cdot, w_1),\cdots,T(\cdot, w_c)) : V_{a,b} \to V_{a'',b''}^{\oplus c}$ is surjective. 
\end{lemma}

\begin{proof}
Let $P\in G(c, V_{a',b'})$ be the span of $w_1,\cdots, w_c$. 
The conditions (ii) and (iii) mean that $v$ is contained in the domain $U$ of regularity of $\varphi$ 
with $\varphi(v)=P$, whence $U\ne\emptyset$. 
Then (iv) implies that the fiber of the morphism $\varphi\colon U\to G(c, V_{a',b'})$ over $P$ 
has the expected dimension ${\dim}V_{a,b}-{\dim}\, G(c, V_{a',b'})$. 
Hence $\varphi(U)$ has dimension $\geq{\dim}\, G(c, V_{a',b'})$, and so $\varphi$ is dominant. 
\end{proof}


\section{Some stable rationality}\label{sec:stable rationality}

A variety $X$ is said to be \textit{stably rational of level $N$} if $X\times{\proj}^N$ is rational. 
In this section we prepare stable rationality results for some quotients of Grassmannians, 
to which the proof of Theorem \ref{main} will be finally reduced. 
We set $\overline{G}={\SLSL}/(-1, -1)$. 
When $a, b>0$ are odd, the element $(-1, -1)$ of ${\SLSL}$ acts on $V_{a,b}$ trivially 
so that $\overline{G}$ acts on $V_{a,b}$. 
This linear $\overline{G}$-action is almost free if ${\PGLPGL}$ acts on ${\proj}V_{a,b}$ almost freely, 
that is, general bidegree $(a, b)$ curves on ${\proj}^1\times{\proj}^1$ have no non-trivial stabilizer. 

\begin{lemma}\label{stable rational basic}
The group $\overline{G}$ acts on $V_{1,1}^{\oplus3}$ almost freely 
with the quotient $V_{1,1}^{\oplus3}/{\SLSL}$ rational. 
\end{lemma}

\begin{proof}
The first assertion follows from the almost freeness of the ${\PGLPGL}$-action on $({\proj}V_{1,1})^3$. 
For the second assertion, we first note that 
\begin{equation*}
V_{1,1}^{\oplus3}/{\SLSL} \sim (V_{1,1}^{\oplus3}/{\GLGL})\times{\C}^{\times}. 
\end{equation*}
The group ${\GLGL}$ acts on $V_{1,1}$ almost transitively 
with the stabilizer of a general point isomorphic to ${\GL}_2$ 
(identify $V_{1,1}$ with ${\hom}(V_1, V_1)$). 
Hence, applying the slice method (cf.~\cite{Do}) to the first projection $V_{1,1}^{\oplus3}\to V_{1,1}$, 
we obtain 
\begin{equation*}
V_{1,1}^{\oplus3}/{\GLGL} \sim V_{1,1}^{\oplus2}/{\GL}_2, 
\end{equation*}
where ${\GL}_2$ acts on $V_{1,1}^{\oplus2}$ linearly in the right hand side. 
Then the quotient $V_{1,1}^{\oplus2}/{\GL}_2$ is rational by the result of Katsylo \cite{Ka2}. 
\end{proof}

\begin{corollary}\label{linear system stable rational}
Let $n>0$ be an odd number. 
Then ${\proj}V_{1,n}/{\SLSL}$ and ${\proj}V_{3,n}/{\SLSL}$ are stably rational of level $13$. 
\end{corollary}

\begin{proof}
We treat the case of $V_{1,n}$. 
For dimensional reason we may assume $n>3$. 
Then the group $\overline{G}$ acts on $V_{1,n}$ almost freely. 
Hence we may apply the no-name lemma to both projections 
$V_{1,1}^{\oplus3}\oplus V_{1,n} \to V_{1,n}$ and 
$V_{1,1}^{\oplus3}\oplus V_{1,n}\to V_{1,1}^{\oplus3}$ to see that 
\begin{equation}\label{eqn:prf Cor 3.2}
(V_{1,n}/{\SLSL})\times{\C}^{12} \sim (V_{1,1}^{\oplus3}/{\SLSL})\times{\C}^{2n+2}. 
\end{equation}
By Lemma \ref{stable rational basic}, $V_{1,n}/{\SLSL}$ is stably rational of level $12$. 
Since $V_{1,n}/{\SLSL}$ is birational to ${\C}^{\times}\times({\proj}V_{1,n}/{\SLSL})$, 
our assertion is proved. 
The case of $V_{3,n}$ is similar: 
just replace $V_{1,n}$ by $V_{3,n}$ in this argument, now with $n>1$. 
The only change is that the factor ${\C}^{2n+2}$ in \eqref{eqn:prf Cor 3.2} is replaced by ${\C}^{4n+4}$.  
\end{proof}

\begin{proposition}\label{Grassmann stable rational}
When $n>1$ is odd, $G(3, V_{3,n})/{\SLSL}$ is stably rational of level $2$. 
\end{proposition}

\begin{proof}
Let $\mathcal{F}\to G(3, V_{3,n})$ be the universal sub vector bundle of rank $3$, 
on which ${\SLSL}$ acts equivariantly. 
The elements $(\pm1, \mp1)\in{\SLSL}$ act on $\mathcal{F}$ by multiplication by $-1$. 
Since $\mathcal{F}$ has odd rank, 
they act on the line bundle ${\det}\mathcal{F}$ also by $-1$. 
Hence the bundle $\mathcal{F}'=\mathcal{F}\otimes{\det}\mathcal{F}$ is ${\PGLPGL}$-linearized. 
Note that ${\proj}\mathcal{F}$ is canonically identified with ${\proj}\mathcal{F}'$. 
Since ${\PGLPGL}$ acts on $G(3, V_{3,n})$ almost freely, 
we can apply the no-name lemma to $\mathcal{F}'$ to see that 
\begin{equation*}
{\proj}\mathcal{F}/{\SLSL} \sim {\proj}\mathcal{F}'/{\SLSL} \sim {\proj}^2\times(G(3, V_{3,n})/{\SLSL}). 
\end{equation*}
Thus it suffices to show that ${\proj}\mathcal{F}/{\SLSL}$ is rational. 

Regarding ${\proj}\mathcal{F}$ as an incidence in $G(3, V_{3,n})\times{\proj}V_{3,n}$, 
we have second projection ${\proj}\mathcal{F}\to{\proj}V_{3,n}$. 
Its fiber over ${\C}l\in{\proj}V_{3,n}$ is 
the sub Grassmannian in $G(3, V_{3,n})$ of $3$-planes containing ${\C}l$, 
and hence identified with $G(2, V_{3,n}/{\C}l)$. 
Therefore, if $\mathcal{G}\to{\proj}V_{3,n}$ is the universal quotient bundle of rank ${\dim}V_{3,n}-1$, 
then ${\proj}\mathcal{F}$ is identified with the relative Grassmannian $G(2, \mathcal{G})$. 
The elements $(\pm1, \mp1)\in{\SLSL}$ act on $\mathcal{G}$ by multiplication by $-1$, 
and also on $\mathcal{O}_{{\proj}V_{3,n}}(1)$ by $-1$. 
Thus the bundle 
$\mathcal{G}'=\mathcal{G}\otimes\mathcal{O}_{{\proj}V_{3,n}}(1)$ 
is ${\PGLPGL}$-linearized, 
and $G(2, \mathcal{G})$ is canonically isomorphic to $G(2, \mathcal{G}')$. 
Since ${\PGLPGL}$ acts on ${\proj}V_{3,n}$ almost freely, 
we can use the no-name lemma to trivialize the ${\PGLPGL}$-bundle $\mathcal{G}'$ 
locally in the Zariski topology. 
Hence we have 
\begin{equation*}
G(2, \mathcal{G}')/{\SLSL} \sim G(2, {\C}^{4n+3})\times({\proj}V_{3,n}/{\SLSL}). 
\end{equation*}
Since ${\dim}\, G(2, {\C}^{4n+3}) > 13$ for $n>1$, 
our assertion follows from Corollary \ref{linear system stable rational}. 
\end{proof}

We also treat $G(3, V_{3,1})$ which is excluded above. 

\begin{proposition}\label{Grassmann stable rational II}
The quotient $G(3, V_{3,1})/{\SLSL}$ is stably rational of level $5$. 
\end{proposition}

\begin{proof}
In this case the ${\PGLPGL}$-action on ${\proj}V_{3,1}$ is not almost free, 
having the Klein 4-group as a general stabilizer, 
so that we cannot apply the above proof.  But the following modification will work: 
replace $\mathcal{F}$ with $\mathcal{F}^{\oplus2}$, 
and the projection ${\proj}\mathcal{F}\to{\proj}V_{3,1}$ with  
\begin{equation*}
{\proj}(\mathcal{F}^{\oplus2}) \to {\proj}(V_{3,1}^{\oplus2}), \qquad 
(P, {\C}(v_1, v_2))\mapsto{\C}(v_1, v_2), 
\end{equation*}
where $v_1, v_2\in V_{3,1}$ are vectors contained in the $3$-plane $P$. 
Then we can imitate the above argument to deduce that 
\begin{equation*}
{\proj}(\mathcal{F}^{\oplus2})/{\SLSL} \sim {\proj}^5\times(G(3, V_{3,1})/{\SLSL}), 
\end{equation*}
\begin{equation*}
{\proj}(\mathcal{F}^{\oplus2})/{\SLSL} \sim {\proj}^5\times({\proj}(V_{3,1}^{\oplus2})/{\SLSL}). 
\end{equation*}
Thus it suffices to prove that 
${\proj}(V_{3,1}^{\oplus2})/{\SLSL}$ is stably rational of level $5$. 

Consider the representation 
$W=V_{1,1}\oplus V_{3,1}^{\oplus2}$. 
We apply the no-name lemma to both projections 
${\proj}W\dashrightarrow{\proj}(V_{3,1}^{\oplus2})$ and ${\proj}W\dashrightarrow{\proj}(V_{1,1}\oplus V_{3,1})$ 
to see that 
\begin{equation*}
{\C}^4\times({\proj}(V_{3,1}^{\oplus2})/{\SLSL}) \sim 
{\C}^8\times({\proj}(V_{1,1}\oplus V_{3,1})/{\SLSL}). 
\end{equation*}
Using the slice method for the projection $V_{1,1}\oplus V_{3,1}\to V_{1,1}$, 
we then have 
\begin{equation*}
(V_{1,1}\oplus V_{3,1})/{\GLGL} \sim V_{3,1}/{\GL}_2. 
\end{equation*}
Finally, $V_{3,1}/{\GL}_2$ is rational by Katsylo \cite{Ka2}. 
\end{proof}


\section{Proof of Theorem \ref{main}}\label{sec:proof}

Let $b\geq5$ be an odd number. 
In this section we prove that ${\proj}V_{3,b}/{\SLSL}$ is rational (Theorem \ref{main}) 
by executing the method of double bundle explained in \S \ref{ssec:double bundle}. 
In logical order, the proof proceeds in the following line. 

\begin{enumerate}
\item We choose a bi-transvectant 
\begin{equation*}
T=T^{(r,s)}: V_{3,b}\times V_{a',b'} \to V_{a'',b''}
\end{equation*} 
according to Table \ref{value (r, s) etc} below. 
This satisfies that  
$c:={\dim}V_{a',b'} - {\dim}V_{a'',b''}$ is either $1$ or $3$, 
${\dim}V_{3,b}>c\cdot{\dim}V_{a'',b''}$, 
and that both $a'$ and $b'$ are odd. 
\item We check that $T$ satisfies the non-degeneracy condition \eqref{eqn:non-degeneracy} 
by finding vectors $v\in V_{3,b}$, $w_1,\cdots,w_c\in V_{a',b'}$ as in Lemma \ref{non-degeneracy}. 
\item Then, as shown in \S \ref{ssec:double bundle}, $V_{3,b}$ gets birationally realized as 
an ${\SLSL}$-linearized vector bundle $\mathcal{E}$ over $G(c, V_{a',b'})$ 
which is a sub bundle of $V_{3,b}\times G(c, V_{a',b'})$. 
(In case $c=1$, $G(c, V_{a',b'})$ is just ${\proj}V_{a',b'}$.)  
\item Since $3$ and $b$ are odd, the elements $(\pm1, \mp1)\in{\SLSL}$ 
act on $\mathcal{E}$ by multiplication by $-1$. 
\item Since $a'$ and $b'$ are odd, 
$(\pm1, \mp1)$ act on the universal sub bundle $\mathcal{F}$ over $G(c, V_{a',b'})$ also by $-1$. 
Since $\mathcal{F}$ has odd rank ($=c$), $(\pm1, \mp1)$ act on ${\det}\mathcal{F}$ by $-1$. 
Hence $\mathcal{E}\otimes{\det}\mathcal{F}$ is ${\PGLPGL}$-linearized. 
\item It is not difficult to see that ${\PGLPGL}$ acts on $G(c, V_{a',b'})$ almost freely. 
Then by the no-name lemma we have 
\begin{equation*}
{\proj}\mathcal{E}/{\PGLPGL} \sim {\proj}^N \times (G(c, V_{a',b'})/{\PGLPGL}) 
\end{equation*}
as explained in \eqref{eqn:apply twisted no-name}, 
where $N={\dim}\,{\proj}V_{3,b} - {\dim}\, G(c, V_{a',b'})$. 
\item The quotient $G(c, V_{a',b'})/{\SLSL}$ is stably rational of level $\leq N$ 
by Corollary \ref{linear system stable rational} and 
Propositions \ref{Grassmann stable rational} and \ref{Grassmann stable rational II}. 
(see the values of $c$, $(a', b')$, $N$ below.) 
This concludes that ${\proj}V_{3,b}/{\SLSL}$ is rational. 
\end{enumerate}

The bi-transvectant $T^{(r,s)}$ is provided systematically according to 
the remainder $[b]\in{\Z}/5{\Z}$, except the case $b=7$. 

\begin{proposition}\label{non-degeneracy check}
For odd $b\geq5$ 
we set the values of $(r, s)$ and $(a', b')$ (and hence $(a'', b'')$, $c$ and $N$) 
by the following Table \ref{value (r, s) etc}. 
\begin{table}[h]
\caption{Input of bi-transvectant}\label{value (r, s) etc}
\begin{center} 
\begin{tabular}{c|c|c|c|c|c}
   $b$        & $(r, s)$        & $(a', b')$      & $(a'', b'')$    & $c$ & $N$          \\ \hline  
$5n$         & $(3, n)$       & $(3, n)$       & $(0, 4n)$     & $3$ & $8n$         \\ \hline  
$5n+1$     & $(1, 3n+1)$ & $(1, 3n+1)$ & $(2, 2n)$     & $1$ & $14n+4$  \\ \hline  
$5n+2$     & $(3, n)$       & $(3, n)$       & $(0, 4n+2)$ & $1$ & $16n+8$  \\ \hline  
$5n+3$     & $(3, n)$       & $(3, n+1)$   & $(0, 4n+4)$ & $3$ & $8n$        \\ \hline  
$5n+4$     & $(1, 3n+3)$ & $(1, 3n+4)$ & $(2, 2n+2)$ & $1$ & $14n+10$ \\ \hline 
$7$           & $(2, 3)$       &  $(3, 3)$      & $(2, 4)$       & $1$ & $16$ \\ \hline 
\end{tabular}
\end{center}
\end{table}
Here $n$ is even when $b\equiv1, 3 \; (5)$, odd when $b\equiv0, 2, 4 \; (5)$, 
and $n>1$ when $b\equiv 2 \; (5)$. 
Then the above argument (1), ..., (7) works. 
\end{proposition}

Notice that we have to separate the case $b=7$ because the ${\PGLPGL}$-action on 
$G(c, V_{a',b'})={\proj}V_{3,1}$ is not almost free, 
so that the step (6) would not work with $(r, s)=(a', b')=(3, 1)$.

For the proof of Proposition \ref{non-degeneracy check}, we are now only left with the step (2) to fill out. 
In the remainder of the article 
we choose vectors $v\in V_{3,b}$ and $w_1,\cdots,w_c\in V_{a',b'}$ 
that should satisfy the conditions (i), ..., (iv) of Lemma \ref{non-degeneracy}.  
In any case the equality $T(v, w_i)=0$ (the condition (ii)) can be checked with a direct calculation 
using the formula of $T=T^{(r,s)}$ given in \S \ref{ssec:transvectant}. 
We leave this to the reader. 
The linear independence of $w_1,\cdots,w_c$ (the condition (i)) can be seen at a glance, and we also omit it. 
Note that this is even trivial when $c=1$. 
Thus what we are going to verify below is the surjectivity conditions (iii) and (iv).  

We shall use the notation $([x, y], [X, Y])$ for the bi-homogeneous coordinate of ${\proj}^1\times{\proj}^1$. 
Thus elements of $V_{a,b}$ will be expressed as 
\begin{equation*}
\sum_{i}F_i(x, y)G_i(X, Y),  
\end{equation*}
where $F_i$, $G_i$ are binary forms of degree $a$, $b$ respectively.


\subsection{The case $b\equiv0 \; (5)$}\label{ssec:0}

We take vectors $v\in V_{3,5n}$, $\vec{w}=(w_1, w_2, w_3)\in (V_{3,n})^3$ by 
\begin{equation*}
v = \binom{5n}{n}X^nY^{4n}x^3 + 3\binom{5n}{2n}X^{2n}Y^{3n}x^2y + 
       3\binom{5n}{2n}X^{3n}Y^{2n}xy^2 + \binom{5n}{n}X^{4n}Y^ny^3,         
\end{equation*}
\begin{eqnarray*}
w_1 &=& Y^nx^3 - X^nx^2y,  \\
w_2 &=& Y^nx^2y - X^nxy^2,  \\
w_3 &=& Y^nxy^2 - X^ny^3. 
\end{eqnarray*}

The map $T(v,\cdot)\colon V_{3,n}\to V_{0,4n}$ is surjective because 
\begin{equation*}
T(v, V_nx^3) = {\C}\langle X^{4n}, \cdots, X^{3n}Y^n\rangle,  \quad 
T(v, V_nx^2y) = {\C}\langle X^{3n}Y^n, \cdots, X^{2n}Y^{2n}\rangle, 
\end{equation*}
\begin{equation*}
T(v, V_nxy^2) = {\C}\langle X^{2n}Y^{2n}, \cdots, X^nY^{3n}\rangle,  \quad 
T(v, V_ny^3) = {\C}\langle X^nY^{3n}, \cdots, Y^{4n}\rangle.  
\end{equation*}

To see the surjectivity of $T(\cdot, \vec{w})\colon V_{3,5n}\to V_{0,4n}^{\oplus3}$, 
we note that 
\begin{equation*}
T(V_{5n}x^3\oplus V_{5n}y^3, \vec{w}) = (V_{0,4n}, 0, V_{0,4n}) \subset V_{0,4n}^{\oplus3}. 
\end{equation*}
Since $T(V_{5n}x^2y, w_2) = V_{0,4n}$, 
then $(0, V_{0,4n}, 0)\subset V_{0,4n}^{\oplus3}$ is also contained in 
the image of $T(\cdot,\vec{w})$.


\subsection{The case $b\equiv1 \; (5)$}\label{ssec:1}

We take the following vectors of $V_{3,5n+1}$ and $V_{1,3n+1}$: 
\begin{eqnarray*}\notag
v  &=&  \binom{5n+1}{2n}X^{3n+1}Y^{2n}x^3 + 3\binom{5n+1}{n}X^{4n+1}Y^{n}x^2y \\
    & &  \quad + 3\binom{5n+1}{2n}X^{2n}Y^{3n+1}xy^2 + \binom{5n+1}{n}X^{n}Y^{4n+1}y^3,  \\
w &=& (X^{3n+1}-Y^{3n+1})x - (X^nY^{2n+1}-X^{2n+1}Y^n)y. 
\end{eqnarray*} 

We shall prove the surjectivity of $T(v, \cdot)\colon V_{1,3n+1}\to V_{2,2n}$ 
by showing that its kernel is $1$-dimensional. 
Suppose we have a vector $w'=G_+(X, Y)x+G_-(X, Y)y$ in $V_{1,3n+1}$ with $T(v, w')=0$. 
Then we have 
\begin{eqnarray*}
T^{(3n+1)}(X^{n}Y^{4n+1}, G_+)   &=& b_0T^{(3n+1)}(X^{2n}Y^{3n+1}, G_-),   \\
T^{(3n+1)}(X^{2n}Y^{3n+1}, G_+) &=& b_1T^{(3n+1)}(X^{4n+1}Y^{n}, G_-),   \\
T^{(3n+1)}(X^{4n+1}Y^{n}, G_+)   &=& b_2T^{(3n+1)}(X^{3n+1}Y^{2n}, G_-). 
\end{eqnarray*}
for suitable constants $b_j$. 
Expanding 
$G_{\pm}=\sum_{i}\alpha_{i}^{\pm}X^{3n+1-i}Y^{i}$, 
we obtain 
\begin{eqnarray*}
\alpha_i^+ = c_{1i}\alpha_{i+n}^- \: \: (0\leq i\leq n),           &  &   \alpha_i^- = 0 \: \: (0\leq i\leq n-1),  \\
\alpha_i^+ = c_{2i}\alpha_{i+2n+1}^{-} \: \: (0\leq i\leq n),   &  &   \alpha_i^+ = 0 \: \: (n+1\leq i\leq 2n),   \\
\alpha_{i+n}^+ = c_{3i}\alpha_i^- \: \: (n+1\leq i\leq 2n+1),  &  &  \alpha_i^- = 0 \: \: (2n+2\leq i\leq 3n+1),   
\end{eqnarray*}
for some fixed constants $c_{\ast}$. 
This reduces to the relations 
\begin{equation*}
\alpha_0^+ = d_1\alpha_{3n+1}^+ = d_2\alpha_{n}^- = d_3\alpha_{2n+1}^-
\end{equation*} 
where $d_j$ are appropriate constants, 
and $\alpha_{i}^{\pm}=0$ for other $i$. 
This implies our assertion. 

The surjectivity of  
$T(\cdot, w)\colon V_{3,5n+1}\to V_{2,2n}$ 
can be seen by noticing that 
\begin{equation*}
T(V_{5n+1}y^3, w) = V_{2n}y^2, \qquad T(V_{5n+1}x^3, w) = V_{2n}x^2, 
\end{equation*}
\begin{equation*}
T(V_{5n+1}xy^2, (X^{3n+1}-Y^{3n+1})x) = V_{2n}xy. 
\end{equation*}


\subsection{The case $b\equiv2 \; (5)$}\label{ssec:2}

We take vectors in $V_{3,5n+2}$ and $V_{3,n}$ by  
\begin{equation*}
v = X^{n}Y^{4n+2}x^3 + X^{2n+1}Y^{3n+1}x^2y + X^{3n+1}Y^{2n+1}xy^2 + X^{4n+2}Y^{n}y^3, 
\end{equation*}
\begin{equation*}
w = Y^nx^2y - X^nxy^2. 
\end{equation*} 

The map $T(v, \cdot)\colon V_{3,n}\to V_{0,4n+2}$ is surjective because 
\begin{eqnarray*}
T(v, V_nx^3) &=& {\C}\langle X^{4n+2}, \cdots, X^{3n+2}Y^n\rangle, \\
T(v, V_nx^2y) &=& {\C}\langle X^{3n+1}Y^{n+1}, \cdots, X^{2n+1}Y^{2n+1}\rangle,  \\
T(v, V_nxy^2) &=& {\C}\langle X^{2n+1}Y^{2n+1}, \cdots, X^{n+1}Y^{3n+1}\rangle,  \\
T(v, V_ny^3) &=& {\C}\langle X^{n}Y^{3n+2}, \cdots, Y^{4n+2}\rangle. 
\end{eqnarray*}

On the other hand, we have 
$T(V_{5n+2}xy^2, w)=V_{0,4n+2}$ 
so that the map 
$T(\cdot, w):V_{3,5n+2}\to V_{0,4n+2}$ 
is also surjective.


\subsection{The case $b\equiv3 \; (5)$}\label{ssec:3}

We take the following vectors of $V_{3,5n+3}$ and $V_{3,n+1}$ 
according to the remainder of $n$ modulo $5$: 

\noindent
(1) When $n\nequiv4$ mod $5$, we set 
\begin{eqnarray*}\notag
v  &=&  \binom{5n+3}{n}X^{n}Y^{4n+3}x^3 + \binom{5n+3}{2n+1}X^{2n+1}Y^{3n+2}x^2y \\
    & &  + \binom{5n+3}{2n+1}X^{3n+2}Y^{2n+1}xy^2 + \binom{5n+3}{n}X^{4n+3}Y^{n}y^3,  \\
w_1 &=& X^{n+1}y^3 + Y^{n+1}xy^2, \\
w_2 &=& X^{n+1}xy^2 + Y^{n+1}x^2y,  \\
w_3 &=& X^{n+1}x^2y + Y^{n+1}x^3.  
\end{eqnarray*} 

\noindent
(2) When $n\equiv4$ mod $5$, we denote $n=2m$ (remember $n$ is even) and set 
\begin{eqnarray*}\notag
v  &=&  \left\{ \frac{7m+3}{m+1}\frac{5m+2}{3m+2}\binom{5n+3}{m}X^{m}Y^{9m+3} + X^{9m+5}Y^{m-2}\right\} x^3 \\ 
    & &  + 3\frac{5m+2}{3m+2}\binom{5n+3}{3m+1}X^{3m+1}Y^{7m+2}x^2y  + 
         3\binom{5n+3}{5m+2}X^{5m+2}Y^{5m+1}xy^2  \\
       & &  + \frac{5m+3}{3m+1}\binom{5n+3}{7m+3}X^{7m+3}Y^{3m}y^3,  
\end{eqnarray*} 
and use the same $w_i$ as above. 

When $n\nequiv4$ mod $5$, we have no $0\leq j\leq n+1$ with $j(5n+5)=(i+1)(n+1)$ 
for $i=n$, $2n+1$, $3n+2$, $4n+3$. 
Hence by Lemma \ref{pre-apolar non-dege}, for those $i$ the bilinear map 
\begin{equation}\label{eqn:pre-apolar b==4}
T^{(n)} : {\C}X^iY^{5n+3-i} \times {\C}X^{n+1-j}Y^j \to {\C}X^{i-j+1}Y^{4n+3-i+j} 
\end{equation}
is non-degenerate for any $j$, as far as the indices are non-negative. 
It follows that 
\begin{eqnarray*}\notag
T(v, V_{n+1}x^3)  &=& {\C}\langle X^{4n+4}, \cdots, X^{3n+3}Y^{n+1} \rangle,  \\
T(v, V_{n+1}x^2y) &=& {\C}\langle X^{3n+3}Y^{n+1}, \cdots, X^{2n+2}Y^{2n+2} \rangle,  \\
T(v, V_{n+1}xy^2)  &=& {\C}\langle X^{2n+2}Y^{2n+2}, \cdots, X^{n+1}Y^{3n+3} \rangle,  \\
T(v, V_{n+1}y^3)  &=& {\C}\langle X^{n+1}Y^{3n+3}, \cdots, Y^{4n+4} \rangle,  
\end{eqnarray*}
whence the map 
$T(v,\cdot):V_{3,n+1}\to V_{0,4n+4}$ is surjective. 
We leave it to the reader to check similar surjectivity when $n\equiv4 \; (5)$. 
In that case, since $m\equiv2\;(5)$, we have no $j$ with $j(5n+5)=(i+1)(n+1)$ for 
$i=m+k(n+1)$, $0\leq k\leq3$, and $i=9m+5$. 
Hence for those $i$ the map \eqref{eqn:pre-apolar b==4} is non-degenerate for any relevant $j$, 
again by Lemma \ref{pre-apolar non-dege}. 

To see that 
\begin{equation*}
T(\cdot, \vec{w}) = (T(\cdot, w_1),T(\cdot, w_2), T(\cdot, w_3)) : V_{3,5n+3}\to V_{0,4n+4}^{\oplus3} 
\end{equation*}
is surjective (regardless of $[n]\in{\Z}/5{\Z}$), 
we note that the bilinear maps 
\begin{equation*}
T^{(n)}(\cdot, X^{n+1}):{\C}X^iY^{5n+3-i}\to{\C}X^{i+1}Y^{4n+3-i} 
\end{equation*} 
\begin{equation*}
T^{(n)}(\cdot, Y^{n+1}):{\C}X^iY^{5n+3-i}\to{\C}X^{i-n}Y^{5n+4-i} 
\end{equation*} 
are non-degenerate whenever the indices are non-negative. 
It follows that 
\begin{equation*}
T(V_{5n+3}x^3, \vec{w}) = ({\C}\langle X^{4n+4},\cdots, XY^{4n+3}\rangle, 0, 0),  
\end{equation*}
\begin{equation*}
T({\C}\langle X^nY^{4n+3}x^2y, X^{2n+1}Y^{3n+2}xy^2, X^{3n+2}Y^{2n+1}y^3 \rangle, \vec{w}) 
\supset ({\C}Y^{4n+4}, 0, 0),  
\end{equation*}
so that $(V_{0,4n+4}, 0, 0)\subset V_{0,4n+4}^{\oplus3}$ is contained in the image of $T(\cdot, \vec{w})$. 
Similarly, we see that 
$(0, 0, V_{0,4n+4})\subset V_{0,4n+4}^{\oplus3}$ is contained in the image too. 
Finally, since $T(\cdot, w_2)$ maps the space 
$V_{5n+3}x^2y\oplus V_{5n+3}xy^2$ onto $V_{0,4n+4}$, 
we find using the above results that $(0, V_{0,4n+4}, 0)$ is also contained in the image.


\subsection{The case $b\equiv4 \; (5)$}\label{ssec:4}

We take the following vectors of $V_{3,5n+4}$ and $V_{1,3n+4}$: 
\begin{eqnarray*}
v &=& \frac{3n+4}{n+2}\frac{3n+4}{n+1}\binom{5n+4}{2n+1}X^{3n+3}Y^{2n+1}x^3  
           + 3\frac{3n+4}{n+1}\binom{5n+4}{n}X^{4n+4}Y^{n}x^2y  \\ 
   & &  -  3\binom{5n+4}{2n+1}X^{2n+1}Y^{3n+3}xy^2  
           -  \frac{n+2}{3n+4}\binom{5n+4}{n}X^{n}Y^{4n+4}y^3, \\
w &=& (X^{3n+4}+Y^{3n+4})x + (X^{2n+3}Y^{n+1}+X^{n+1}Y^{2n+3})y.   
\end{eqnarray*}
 
We shall show that the kernel of 
$T(v, \cdot)\colon V_{1,3n+4}\to V_{2,2n+2}$ is $1$-dimensional, 
which then implies its surjectivity. 
We first note that $5n+6$ and $3n+4$ are coprime by the Euclidean algorithm. 
By Lemma \ref{pre-apolar non-dege}, the bilinear map 
\begin{equation*}
T^{(3n+3)} : {\C}X^{i}Y^{5n+4-i} \times {\C}X^{3n+4-j}Y^{j} \to {\C}X^{i-j+1}Y^{2n+1-i+j}
\end{equation*}
is non-degenerate whenever the indices are non-negative. 
Now suppose a vector $w'=G_+(X, Y)x+G_-(X, Y)y$ in $V_{1,3n+4}$ satisfies $T(v, w')=0$. 
This is rewritten as 
\begin{eqnarray*}
T^{(3n+3)}(X^{3n+3}Y^{2n+1}, G_-) &=& b_0T^{(3n+3)}(X^{4n+4}Y^{n}, G_+),  \\ 
T^{(3n+3)}(X^{4n+4}Y^{n}, G_-)       &=& b_1T^{(3n+3)}(X^{2n+1}Y^{3n+3}, G_+),  \\ 
T^{(3n+3)}(X^{2n+1}Y^{3n+3}, G_-) &=& b_2T^{(3n+3)}(X^{n}Y^{4n+4}, G_+),  
\end{eqnarray*}
for some constants $b_j$. 
Expanding 
$G_{\pm}(X, Y)=\sum_{j=0}^{3n+4}\alpha_{j}^{\pm}X^{3n+4-j}Y^{j}$, 
we obtain the relation 
\begin{eqnarray*}\notag
\alpha_{j+n+1}^+  =  c_{1j}\alpha_j^- \: \: \:  (n+2 \leq j \leq 2n+3),  &  &  \alpha_j^- = 0 \: \: \:   (2n+4 \leq j \leq 3n+4),  \\ 
\alpha_j^+ =  c_{2j}\alpha_{j+2n+3}^- \: \: \:   (0 \leq j \leq n+1),      &  & \alpha_j^+ = 0 \: \: \:   (n+2 \leq j \leq 2n+2),   \\ 
\alpha_j^+  =  c_{3j}\alpha_{j+n+1}^-  \: \: \:   (0 \leq j \leq n+1),  &  & \alpha_j^- = 0 \: \: \:   (0\leq j \leq n), 
\end{eqnarray*}
where $c_{\ast}$ are suitable non-zero constants. 
This is reduced to the relations 
\begin{equation*}
\alpha_0^+=d_1\alpha_{n+1}^-=d_2\alpha_{2n+3}^-=d_3\alpha_{3n+4}^+
\end{equation*} 
for some constants $d_j$, 
and $\alpha_i^{\pm}=0$ for other $i$. 
This proves our claim. 

On the other hand, the surjectivity of   
$T(\cdot, w)\colon V_{3,5n+4}\to V_{2, 2n+2}$ 
follows by noticing that 
\begin{equation*}
T(V_{5n+4}x^3, w) = V_{2n+2}x^2, \qquad 
T(V_{5n+4}y^3, w) = V_{2n+2}y^2, 
\end{equation*}
\begin{equation*}
T(V_{5n+4}xy^2, (X^{3n+4}+Y^{3n+4})x) = V_{2n+2}xy. 
\end{equation*}


\subsection{The case $b=7$}\label{ssec:7}

We choose the following vectors of $V_{3,7}$ and $V_{3,3}$: 
\begin{eqnarray*}\notag
v &=& \binom{7}{3}X^3Y^4x^3 - 9Y^7x^2y + \binom{7}{1}X^6Yxy^2 + \binom{7}{3}X^4Y^3y^3,  \\ 
w &=& Y^3x^3 + X^3xy^2 + (XY^2+Y^3)y^3. 
\end{eqnarray*}

We leave it to the reader to check that $w$ spans the kernel of  
$T(v, \cdot):V_{3,3}\to V_{2,4}$ 
(cf.~\S \ref{ssec:1} and \S \ref{ssec:4}). 
We shall show that $T(\cdot, w)\colon V_{3,7}\to V_{2,4}$ is surjective too. 
First note that the bilinear map 
\begin{equation*}
T^{(2)} : {\C}x^iy^{3-i}\times{\C}x^{3-j}y^j \to {\C}x^{i-j+1}y^{j-i+1}
\end{equation*}
is non-degenerate whenever the indices are non-negative, 
for $3$ and $5$ are coprime (Lemma \ref{pre-apolar non-dege}).
Then we have 
\begin{equation*}
T(V_7y^3, w) = T(V_7y^3, Y^3x^3) = V_4xy. 
\end{equation*}
Since $T^{(3)}(V_7, X^3)=V_4$, 
we have $T(V_7x^3, w)\subset V_4x^2\oplus V_4xy$ 
with surjective projection $T(V_7x^3, w)\to V_4x^2$.  
Therefore $V_4x^2$ is also contained in the image of $T(\cdot, w)$. 
Finally, since $T(V_7xy^2, X^3xy^2)=V_4y^2$,  
the space $V_4y^2$ is contained in the image too.


\end{document}